\documentclass[leqno,12pt]{amsart}
\usepackage{setspace}
\usepackage[english]{babel}
\usepackage{bussproofs}
\usepackage{csquotes}
\usepackage{dirtytalk}
\usepackage{mathtools}
\usepackage{mathrsfs}
\usepackage{latexsym}
\usepackage{enumitem}
\usepackage{amsthm}
\usepackage{amssymb}
\DeclareMathAlphabet{\mathpzc}{OT1}{pzc}{m}{it}
\usepackage[latin1]{inputenc}
\usepackage{afterpage}

\usepackage{bbm}
\usepackage[rgb,dvipsnames]{xcolor}
\usepackage{tikz} 
\usetikzlibrary{decorations.text}
\usetikzlibrary{arrows,arrows.meta,hobby}
\usetikzlibrary{shapes.misc, positioning}

\usepackage{tikz-cd}

\usetikzlibrary{cd}

\newtheorem{theorem}{Theorem}[section]
\newtheorem{maintheorem}{Main Theorem}[section]

\newtheorem{lemma}[theorem]{Lemma}
\newtheorem{proposition}[theorem]{Proposition}
\newtheorem{corollary}[theorem]{Corollary}
\newtheorem{observation}[theorem]{Observation}
\newtheorem{fact}[theorem]{Fact}

\theoremstyle{definition}
\newtheorem{definition}[theorem]{Definition}

\theoremstyle{remark}
\newtheorem{remark}{Remark}

\newtheorem{question}{Question}

\def\hook{\upharpoonright}
\def\forces{\Vdash}

\def\Me{\mathcal M}
\def\Null{\mathcal N}
\def\ZFC{\mathsf{ZFC}}

\def\PFA{\mathsf{PFA}}
\def\MA{\mathsf{MA}}

\def\baire{\omega^\omega}

\def\cantor{2^\omega}

\def\mfb{\mathfrak b}
\def \mfd{\mathfrak{d}}
\def\GCH {\mathsf{GCH}}
\def\CH {\mathsf{CH}}
\def\Q{\mathbb Q}
\def\P{\mathbb P}

\def\cc{2^{\aleph_0}}
\def\mfa{\mathfrak{a}}

\def\st{\mathfrak{st}}

\def\non{{\rm non}}
\def\cov{{\rm cov}}
\def\mfs{\mathfrak{s}}

\usepackage[margin=1.3in]{geometry}

\begin{document}

\title{The Special Tree Number}

\author[Switzer]{Corey Bacal Switzer}
\address[C.~B.~Switzer]{Institut f\"{u}r Mathematik, Kurt G\"odel Research Center, Universit\"{a}t Wien, Kolingasse 14-16, 1090 Wien, AUSTRIA}
\email{corey.bacal.switzer@univie.ac.at}

\thanks{\emph{Acknowledgments:} The author would like to thank the
Austrian Science Fund (FWF) for the generous support through grant number Y1012-N35.}

\date{}

\maketitle

\begin{abstract}
Define {\em the special tree number}, denoted $\st$, to be the least size of a tree of height $\omega_1$ which is neither special nor has a cofinal branch. This cardinal had previously been studied in the context of fragments of $\MA$ but in this paper we look at its relation to other, more typical, cardinal characteristics. Classical facts imply that $\aleph_1 \leq \st \leq \cc$, under Martin's Axiom $\st = \cc$ and that $\st = \aleph_1$ is consistent with $\MA({\rm Knaster}) + \cc = \kappa$ for any regular $\kappa$ thus the value of $\st$ is not decided by $\ZFC$ and in fact can be strictly below essentially all well studied cardinal characteristics. We show that conversely it is consistent that $\st = \cc = \kappa$ for any $\kappa$ of uncountable cofinality while $\non(\Me) = \mfa = \mathfrak{s} = \mathfrak{g} = \aleph_1$. In particular $\st$ is independent of the lefthand side of Cicho\'{n}'s diagram, amongst other things. The proof involves an in depth study of the standard ccc forcing notion to specialize (wide) Aronszajn trees, which may be of independent interest.
\end{abstract}

\section{Introduction}
A tree $T$ of height $\omega_1$ is called {\em special} if it can be covered by countably many antichains, or, equivalently if it has a {\em specializing function} i.e. a function $f:T \to \omega$ which is injective on chains. In the context of forcing, special trees were first introduced by Baumgartner, Malitz and Reinhardt in \cite{BMR70} to show that $\MA$ actually implies a strengthening of the Souslin hypothesis and have since generated an enormous amount of research in set theory and its peripheries, see e.g. the survey article \cite{HonStej15}. Obviously a special tree cannot contain a cofinal branch and it is a natural question whether the converse is true. In the case that $T$ has countable levels (i.e. is an Aronszajn tree as usually defined) this is a well studied problem that is known to be independent of $\ZFC$. Specifically Baumgartner, Malitz and Reinhardt showed in \cite[Theorem 4]{BMR70} that $\MA + \neg \CH$ implies every tree of height $\omega_1$, cardinality less than $2^{\aleph_0}$ is special while a Souslin tree is a consistent counter example. 

For the most part research has focused on Aronszajn trees and, to a lesser extent {\em wide Aronszajn trees}: trees of height $\omega_1$ with levels of size $\aleph_1$ and no cofinal branch. However, several authors have also considered trees of height $\omega_1$ with no cofinal branch and no a priori assumption on the width of the tree, see e.g. \cite{Tod81}. The most notable case of this {\em Rado's Conjecture} (see \cite{Tod83}): which states that every tree of height $\omega_1$ is either special or contains a subtree of cardinality $\omega_1$ which is not special.

In this paper we look at this general case of trees of height $\omega_1$ from the point of view of cardinal characteristics. Define $\st$ to be the least size of a non-special tree of height $\omega_1$ with no cofinal branch. This cardinal was first\footnote{As far as the author can tell.} mentioned in \cite{whereMA} and studied in more depth in \cite{Morass, LandverTree}. While there are very few papers on $\st$ as a cardinal, a lot of basic information is essentially well known about it. For instance, it is well known that $\st \leq 2^{\aleph_0}$, i.e. $\ZFC$ proves there is a non special tree of size $2^{\aleph_0}$ with no cofinal branch and the aforementioned Baumgartner-Malitz-Reinhardt theorem can be reformulated as the statement that $\MA$ implies $\st = 2^{\aleph_0}$, which follows the heuristic that $\MA$ implies ``all cardinal characteristics are large". It is also essentially a known fact that $\st = \aleph_1$ is consistent with more or less all well-studied cardinal characteristics, in particular those appearing in \cite{BlassHB} and \cite{BarHB}, being arbitrarily large since there may be Souslin trees in a model of $\MA({\rm Knaster})$ with arbitrarily large continuum, see \cite{KenST}. To summarize:

\begin{fact}
\begin{enumerate}
\item
$\ZFC$ proves $\aleph_1 \leq \st\leq2^{\aleph_0}$ with both equalities consistent with the failure of $\CH$.
\item
$\MA$ implies $\st = \cc$.
\item
$\st = \aleph_1$ is consistent with $\MA({\rm Knaster}) + 2^{\aleph_0} = \kappa$ for any regular cardinal $\kappa$.
\end{enumerate}
\label{basicfacts}
\end{fact}

The history of $\st$ is as follows. In \cite{LandverTree}, following a suggestion from the anonymous referee of \cite{whereMA}, the possible values $\st$ were investigated and it was shown that $\st$ could consistently be any regular cardinal ${\leq}\cc$, see \cite[Theorem 2.8]{LandverTree}. Piggybacking off these results, Koszmider nearly completed the picture of the possible values of $\st$ in \cite{Morass} by proving that $\st$ could in fact be any cardinal, singular or regular, of uncountable cofinality\footnote{Note that $\st$ must have uncountable cofinality, see Proposition \ref{COFprop} below.} less than or equal to the cofinality of the continuum, see \cite[Theorem 47]{Morass}. Meanwhile $\st > {\rm cf}(\cc)$ is also consistent by a theorem of Laver as explained in the discussion of Theorem \ref{Laver} below, however much less is known about the value of $\st$ when $\st > {\rm cf}(\cc)$. See Question \ref{Qcof}. Laver's theorem also establishes the consistency of $\st > \non(\Null)$ and hence $\st > \cov(\Me)$, which we erroneously claimed was open in an earlier draft of this paper.

In this article we study more generally the possible behavior of $\st$, and in particular look at how $\st$ compares to other, more well studied cardinal characteristics. Our main theorem is the following.

\begin{maintheorem}
For any $\aleph_1 \leq \kappa \leq {\rm cf}(\mu) \leq \mu$ with $\kappa$ (and $\mu$) of uncountable cofinality it is consistent that $\st = \kappa$, $2^{\aleph_0} = \cov(\Me)=\mu$ and $\non(\Me) = \mfa = \mathfrak{s} =\mathfrak{g}= \aleph_1$.
\label{mainthm1}
\end{maintheorem}

Combining Main Theorem \ref{mainthm1} with Fact \ref{basicfacts} and Laver's Theorem \ref{Laver} discussed below, the following is immediate\footnote{See \cite{BlassHB} for the definitions of $\mathfrak{e}$ and $\mathfrak{p}$. The other cardinals will be defined later on in this section.}.

\begin{corollary}
$\st$ is independent of $\mfa$, $\mathfrak{s}$, $\mathfrak{g}$, $\mathfrak{p}$, $\mathfrak{e}$ and both the left hand side and bottom row of the Cicho\'{n} diagram.
\end{corollary}

The model witnessing Main Theorem \ref{mainthm1} is in some sense the obvious one: a finite support iteration of the ccc forcing notions for specializing trees of height $\omega_1$ with no cofinal branch first introduced in \cite{BMR70}. The meat of the proof is therefore computing cardinal characteristics in this model. As a result we also study this specializing forcing notion and the reals it adds in depth. This appears to be one of the first such studies, though see \cite{ChoZa15} for some related results.

The rest of this paper is organized as follows. We finish this section by recalling the cardinal characteristics we will be studying in this paper. In the next section we provide all necessary preliminaries and basic definitions. We also survey known results including those implying Fact \ref{basicfacts} as well as Laver's aforementioned theorem from \cite{Laver1987} and make some more elementary observations about $\st$. In the following section we study the ccc specializing forcing introduced in \cite{BMR70}. In the Section 4 we look at the model obtained by iterating this forcing with finite support and prove Main Theorem \ref{mainthm1} as well as some related results. Section 5 concludes with open questions and a discussion of avenues for future research.

Before concluding this introduction we briefly recall the cardinals we will study in the proceeding discussion. More information about these cardinals can be found in e.g. \cite{BarJu95,BarHB, BlassHB, Hal17}. Let $\Me$ and $\Null$ denote the ideals of meager and null sets respectively on $2^\omega$ (or any other perfect Polish space, it does not matter for our purposes). For $\mathcal I$ equal to either of them recall the following four cardinals.
\begin{enumerate}
\item
The {\em additivity number}, ${\rm add}(\mathcal I)$, is the least size of a set $\mathcal A \subseteq \mathcal I$ whose union is not in $\mathcal I$.
\item
The {\em uniformity number}, ${\rm non}(\mathcal I)$, is the least size of a set $A \subseteq 2^\omega$ not in $\mathcal I$.
\item
The {\em covering number}, ${\rm cov}(\mathcal I)$ is the least size of a set $\mathcal A \subseteq \mathcal I$ so that $\bigcup \mathcal A = 2^\omega$.
\item
The {\em cofinality number}, ${\rm cof}(\mathcal I)$ is the least size of a set $\mathcal A \subseteq \mathcal I$ so that for every $B \in \mathcal I$ there is an $A \in \mathcal A$ with $B \subseteq A$.
\end{enumerate}

The four numbers above for $\Me$ and $\Null$ alongside the well studied {\em bounding} and {\em dominating} numbers\footnote{See \cite{BlassHB} for definitions of $\mfb$ and $\mfd$. We will not use these numbers here so we omit their definitions.}, $\mfb$ and $\mfd$, fit into a diagram of provable implications known as Cicho\'{n}'s diagram, see \cite[Chapter 2]{BarJu95}. This diagram is pictured as Figure 1 below.

\begin{figure}[h]\label{Figure.Cichon}
\centering
  \begin{tikzpicture}[scale=1.5,xscale=2]
     \draw 
	 (2,1) node (addm) {$add(\Me)$}
           (4,1) node (nonn) {$non (\Null)$}
           (3,3) node (cofm) {$cof (\Me)$}
           (2,2) node (bbb) {$\mfb$}
           (3,2) node (ddd) {$\mfd$}
	(1, 3) node (covn) {$cov (\Null)$}
	(1, 1) node (addn) {$add (\Null)$}
	(4, 3) node (cofn) {$cof (\Null)$}
	(0, 1) node (aleph1) {$\aleph_1$}
          (2, 3)  node (nonm) {$non (\Me)$}
          (5, 3) node (ccc) {$\cc$}
          (3, 1) node (covm) {$cov (\Me)$}
	    
           ;
     \draw[->,>=stealth]
	      (aleph1) edge (addn)
            (addn) edge (addm)
            (addm) edge (covm)
            (covm) edge (nonn)
            (addn) edge (covn)
            (addm) edge (bbb)
            (bbb) edge (ddd)
	     (covm) edge (ddd)
            (ddd) edge (cofm)
            (nonm) edge (cofm)
           (cofm) edge (cofn)
	    (cofn) edge (ccc)
           (nonn) edge (cofn)
           (covn) edge (nonm)
           (bbb) edge (nonm)
            
            ;
  \end{tikzpicture}
\caption{The Cicho\'n Diagram. $\mathfrak{x} \to \mathfrak{y}$ means $\ZFC \vdash \mathfrak{x} \leq \mathfrak{y}$}
\end{figure}
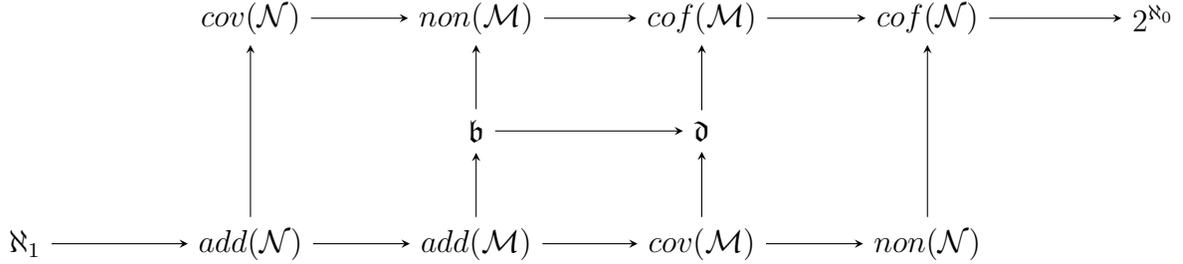

We will also study ``combinatorial" cardinal characteristics of the continuum. These are defined below.

\begin{definition}

\begin{enumerate}
\item
Two sets $A, B \in [\omega]^\omega$ are {\em almost disjoint} if $|A \cap B| < \omega$. A family $\mathcal A \subseteq [\omega]^\omega$ is {\em almost disjoint} if its elements are pairwise almost disjoint and such a family is {\em maximal} or MAD if it is not properly contained in any other almost disjoint family. The {\em almost disjointness number} $\mfa$ is the least size of an infinite MAD family.

\item
Given two sets $A, B \in [\omega]^\omega$ we say that $A$ {\em splits} $B$ if $A \cap B$ and $B \setminus A$ are both infinite. A family $\mathcal S \subseteq [\omega]^\omega$ is {\em splitting} if for every $Y \in [\omega]^\omega$ there is at least one $A \in \mathcal S$ which splits it. The {\em splitting number} $\mathfrak{s}$ is the least size of a splitting family.

\item
A family $\mathcal G \subseteq [\omega]^\omega$ is {\em groupwise dense} if it is downwards closed under almost subsets and given any strictly increasing $f \in \baire$ there is an infinite $A \subseteq \omega$ so that $\bigcup_{k \in A} [f(k), f(k+1)) \in \mathcal G$. The {\em groupwise density number} $\mathfrak{g}$ is the least size of a set of groupwise dense families with empty intersection.
\end{enumerate}

\end{definition}

\bigskip

\noindent {\em Acknowledgments}. The author would like to thank Jeffrey Bergfalk, David Chodounsk\'{y}, Vera Fischer, Chris Lambie-Hanson, Rahman Mohammadpour and Thilo Weinert for some very helpful discussions on this material. In an earlier draft of this paper it was erroneously claimed that $\st$ had not been studied before. Piotr Koszmider kindly sent me the articles \cite{Morass, whereMA, LandverTree} thus straightening out my hubris. Therefore an extra thanks Professor Koszmider for these references as well as his kind explanation of the material therein. Finally thanks to the anonymous referee for a careful reading and some helpful comments.

\section{Introducing the Special Tree Number}

In this section we introduce some terminology and survey some basic facts. Most of our notation is standard, conforming to e.g. the texts \cite{BarJu95, Hal17, JechST, KenST}. To begin we set some vocabulary for trees. Recall that a {\em tree} $T = (T, \sqsubseteq_T)$ is a partial order with the property that for each $t \in T$ the set of strict predecessors of $t$, denoted $P_t$ is well ordered by $\sqsubseteq_T$. For an ordinal $\alpha$ we denote by ${\rm Lev}_\alpha(T)$ the set of $t \in T$ so that $P_t$ has order type $\alpha$. The {\em height} of $T$ is the least $\alpha$ so that ${\rm Lev}_\alpha(T)$ is empty. In this paper we will only be considering trees of height ${\leq}\omega_1$. A tree $T$ is {\em normal} if every $t \in T$ is comparable with some $s \in {\rm Lev}_\alpha(T)$ for each $\alpha$ less than the height of $T$ and given any two nodes $t, s \in {\rm Lev}_{\delta}(T)$ for some limit ordinal $\delta$ if $P_t = P_s$ then $t = s$ i.e. no sequence has two distinct limits. From now on we will assume without further mention that all our trees are normal\footnote{Even if some of them might be special - the arboreal terminology leaves much to be desired.} and, moreover, {\em rooted} i.e. that there is a unique minimal node $\emptyset_T \in T$, called the root. It's well known that restricting our attention to such trees causes no loss of generality for our purposes. 

A {\em branch} through a tree is a maximal, linearly ordered subset. A tree $T$ of height $\omega_1$ is called an {\em Aronszajn tree} if ${\rm Lev}_\alpha(T)$ is countable for each $\alpha < \omega_1$ and there is no uncountable branch. We say that a tree $T$ is a {\em generalized Aronszajn tree} if it is of height $\omega_1$ and has no uncountable branch, with no assumption of the size of levels\footnote{The vocabulary ``Aronszajn tree" is of course standard, while generalized Aroszajn tree is ad hoc vocabulary introduced in this paper so as to circumvent making assumptions about the size of the levels. Vocabulary from the literature along a similar vein include the term {\em wide Aroszajn trees} \cite{DzSh21}, where it is assumed that the levels are of size $\aleph_1$ and $(\omega_1, {\leq}\kappa)$-Aronszajn tree where the levels are assumed to have size $\leq\kappa$, \cite{FuchsMinden}. Since we do not a priori want to assume the levels are of any particular size or bound we avoid these words.}. A tree $T$ of height $\omega_1$ is {\em special} if it can be decomposed into countably many antichains, or equivalently, if it carries a {\em specializing function} i.e. a function $f:T \to \omega$ which is injective on chains. Clearly being special implies being generalized Aronszajn. Note also that if $T$ is special then, considering $T$ as a forcing notion, forcing with $T$ collapses $\omega_1$ since the specializing function becomes an injection from the generic, cofinal branch (of length $\omega_1^V$) into $\omega$.

The main character of this paper is the following cardinal.

\begin{definition}[The Special Tree Number]
The {\em special tree number}, denoted $\st$, is the least size of a non-special generalized Aronszajn tree.
\end{definition}

As noted in the introduction, this cardinal seems to have been first suggested in \cite{whereMA}. It was further developed in \cite{LandverTree}, though not given a name and later called $\sigma$ in \cite{Morass}\footnote{The author was not aware of the competing notation for this idea until after a first draft of this paper had appeared online. At Professor Koszmider's encouragement we keep the notation $\st$ as it lines up with contemporary notation for cardinal characteristics.}. The following three well known results are pertinent to our discussion.

\begin{theorem}[Baumgartner, Malitz and Reinhardt \cite{BMR70}]
Under $\MA + \neg \CH$ every generalized Aronszajn tree of size ${<}2^{\aleph_0}$ is special.
\end{theorem}

A proof of this result can be found in \cite[Theorem 4]{BMR70} or, for a more modern presentation see \cite[Theorem 16.17]{JechST}. However, let us note the point is really that, for any generalized Aronszajn tree $T$ the forcing notion $\P(T)$ consisting of finite, partial functions $p:T \to \omega$ which are injective on chains is ccc. We will discuss this forcing notion in more depth in Sections 3 and 4.

\begin{proposition}
In $\ZFC$ there is always a non-special generalized Aronszajn tree of cardinality $\cc$.
\label{T(S)}
\end{proposition}

\begin{proof}
There are many constructions, the original is probably due to Kurepa \cite{Kurepa}. See the survey article \cite{Tod84} for more details. For the sake of completeness let us sketch the one from \cite{Tod81}, see that article for more details. Let $S \subseteq \omega_1$ be stationary and co-stationary. Let $T(S)$ be the tree consisting of closed, bounded subsets of $S$ ordered by end extension. Clearly this tree has cardinality $\cc$. Since $S$ is stationary there are such sequences of arbitrary countable length. Since $S$ is co-stationary there is no uncountable branch (since its union would be a club through $S$). Finally it is well known that such a tree is Baire i.e. adds no $\omega$ sequences of ordinals, see \cite[Theorem 23.8]{KenST}, and in particular does not collapse $\omega_1$ hence $T$ is not special.
\end{proof}

Recall that an Aronszajn tree $S$ is {\em Souslin} if it is ccc (and hence not special). It is well known that consistently there are Souslin trees. Also recall a forcing notion $\P$ is {\em Knaster} if every uncountable $A \subseteq \P$ contains an uncountable $B \subseteq A$ of pairwise compatible elements. If $\P$ is Knaster and $S$ is a Souslin tree then $\forces_\P$``$\check{S}$ is Souslin" (see \cite[Lemma V.4.13]{KenST}).

\begin{fact}[Lemma V.4.10 of \cite{KenST}]
Finite support iterations of Knaster forcing notions are Knaster. In particular, if $S$ is Souslin and $\P$ is a finite support iteration of Knaster forcing notions, then $\forces_\P$``$\check{S}$ is Souslin".
\end{fact}

Putting these three results together, the following facts are immediate, thus justifying the definition of $\st$ as a cardinal characteristic.

\begin{proposition}
\begin{enumerate}
\item
$\aleph_1 \leq \st \leq \cc$
\item
If $\MA$ holds then $\st = \cc$.
\item
It is consistent that $\st = \aleph_1 < \cc = \kappa$ for any cardinal $\kappa > \aleph_1$ of uncountable cofinality.
\end{enumerate}
\end{proposition}

The last item can be strengthened significantly.

\begin{observation}
For any regular cardinal $\kappa$ it is consistent that $\MA({\rm Knaster}) + \cc = \kappa$ holds and there is a Souslin tree. In particular it is consistent that $\st = \aleph_1$ in a model where all cardinals in the van Douwen and Cicho\'{n} diagrams are arbitrarily large.
\end{observation}

\begin{proof}
It is well know that one can force $\MA({\rm Knaster}) + \cc = \kappa$ by a finite support iteration of Knaster forcing notions, see \cite[Theorem V.4.12]{KenST}, and hence preserve a fixed Souslin tree. Since all cardinals in the Cicho\'{n} and van Douwen diagrams can be increased by Knaster and indeed $\sigma$-linked forcing (\cite{BlassHB, KenST}) the observation follows.
\end{proof}

\begin{remark}
The anonymous referee astutely points out that in fact we can do even better than $\MA({\rm Knaster})$ here. Indeed there are maximal forcing axioms connected to any Souslin tree. For any fixed Souslin tree $S$ let $\MA(S)$ denote the statement that $S$ is Souslin and forcing axiom holds for all ccc forcing notions which preserve the Souslin-ness of $S$ and idem for $\PFA(S)$ with ``ccc" replaced by ``proper". The former is consistent with the continuum any regular cardinal ${>} \aleph_1$ by \cite{Larson99, LT02} and the latter is consistent by \cite{Miyamoto} (and proves the continuum is $\aleph_2$).  See \cite{LT02, PFA(S)} for applications of these fascinating axioms.
\end{remark}

The main goal of this paper is to show that in fact $\st$ can also be larger than many cardinal characteristics as well and therefore is independent of them. There is one result in the literature implicitly of this form, due to Laver. Recall that for a cardinals $\kappa$ and $\lambda$ the {\em dual random model} (for $\kappa$ and $\lambda$) is the model obtained by adding $\lambda$-many random reals over a model of $\MA + \cc = \kappa$.

\begin{theorem}[Laver, see Page 531 of \cite{Laver1987}]
$\st = \kappa$ in the dual random model for any $\kappa > \aleph_1$ and $\lambda$. In particular, let $\aleph_1 < \kappa < \lambda$ be cardinals with $\kappa$ regular and $\lambda$ of uncountable cofinality. It is consistent that $\non(\Null) = \aleph_1 < \st = \kappa < \cov(\Null) = \lambda$. \label{Laver}
\end{theorem}

A consequence of this theorem is that $\st > \cov(\Me)$ (since $\non(\Null) \geq \cov(\Me)$ in $\ZFC$) and $\st > {\rm cf}(\cc)$ (taking $\lambda$ to have e.g. cofinality $\aleph_1$) are both consistent -  two things which do not follow from the results of this paper and, in an earlier draft were claimed to be open. Since Laver's result in \cite{Laver1987} is emphatically not stated in this language we sketch a proof.

\begin{proof}
It is standard, see \cite[Model 7.6.7]{BarJu95}, that in the dual Random model (for any $\kappa$ and $\lambda$) we have $\non(\Null) = \aleph_1$ since any $\aleph_1$ many random reals are non-null. Meanwhile if $\lambda > \kappa$ is of uncountable cofinality then it is also standard (again see e.g. \cite{BarJu95}) that $\cov(\Null) = \lambda$. Finally Laver shows in \cite{Laver1987} that adding any number of random reals over a model of $\MA + \cc = \kappa$ results in a model where there are no nonspecial generalized Aronszajn trees of size ${<}\kappa$. Thus he shows that $\st \geq \kappa$. Therefore we need to show that there is a non-special tree of size $\kappa$. Work in the model of $\MA + \cc = \kappa$ and let $S\subseteq \omega_1$ be stationary co-stationary (which has size $\kappa$ and will continue to in any ccc forcing extension). In the extension by the Random reals we claim that the tree $T(S)$, as computed in the ground model which hence has size $\kappa$ is not special (that it does not have a branch follows from the fact that the stationarity of both $S$ and its complement are preserved since random forcing is ccc). Indeed we show that if $\mathbb B(\lambda)$ is the forcing to add $\lambda$ many random reals then $\forces_{\mathbb B(\lambda)}$``$\check{T}(\check{S})$ does not collapse $\omega_1$". To see why note that, since $T(S)$ adds no new reals it forces $\dot{\mathbb B}(\lambda) = \check{\mathbb B}(\lambda)$ i.e. the new measure algebra is simply the same as the ground model one. Hence $\mathbb B (\lambda) \times \check{T}(\check{S})$ is the same as first forcing with $T(S)$ and then forcing to add the random reals as defined in the extension. But this latter forcing decomposes as an $\omega$-distributive forcing followed by a ccc forcing so $\omega_1$ is preserved and in particular  $\forces_{\mathbb B(\lambda)}$``$\check{T}(\check{S})$ preserves $\omega_1$" so we are done.
\end{proof}

We finish this section with more more elementary fact about $\st$.

\begin{proposition}[Folklore]
The special tree number has uncountable cofinality.
\label{COFprop}
\end{proposition}

This result appears to have been known by several authors, see e.g. \cite[p.3]{Morass}. However we could not find a proof in the literature and, though easy, it seemed worth having written down.

\begin{proof}
Assume $\kappa$ has countable cofinality, $\{\kappa_n\}_{n < \omega}$ is a strictly increasing cofinal sequence of cardinals less than $\kappa$ and for all $n < \omega$ $\kappa_n < \st$. We will show that $\kappa < \st$. Fix a generalized Aronszajn tree $T$ of cardinality $\kappa$ and partition it into countably many trees $T_n$ (with the induced subordering) so that $|T_n| \leq \kappa_n$ (some of these might have countable height, this is fine). Each one is special by assumption. Let $f_n:T_n \to \omega$ be a specializing function. Since the $T_n$'s are disjoint and cover $T$ the disjoint union of the $f_n$'s is a map from $T$ to countably many disjoint copies of $\omega$. But then this is a function from $T$ to a countable set which is injective on chains and hence is a specializing function.
\end{proof}

As mentioned before, Koszimder has shown that $\st$ can be singular of uncountable cofinality.

\section{The Specializing Forcing}
Fix a generalized Aronszajn tree $T = \langle T, \sqsubseteq_T\rangle$ with root $\emptyset_T$. The point of this section is to study new reals, in fact new countable sequences of ordinals, added after forcing with $\P(T)$, the Baumgartner-Malitz-Reinhardt forcing from \cite{BMR70} to specialize $T$ with finite approximations. Concretely a condition $p \in \P(T)$ if it is a finite, partial function $p:T \to \omega$ so that if $s\sqsubseteq_T t$ and $s, t \in {\rm dom}(p)$ then $p(s) \neq p(t)$. The order is reverse inclusion. This forcing is ccc (\cite[Theorem 3]{BMR70} see also \cite[Lemma 16.19]{JechST}) and adds a specializing function for $T$. Throughout, let $G_T \subseteq \P(T)$ be $V$-generic and let $g_T:= \bigcup G$ be the generic specializing function. Also let $\dot{g}_T$ be the canonical $\P(T)$-name for $g_T$. 

\begin{definition}
Let $A \subseteq T$ be a non-empty, countable, downward closed subset of $T$. The {\em border} of $A$, denoted $\partial A$, is the set of all $s \notin A$ so that every predecessor of $s$ is in $A$ i.e. $\partial A = \{s \in T \setminus A \; | \; P_s \subseteq A\}$. In other words it is the set of minimal elements of $T$ not in $A$.
\end{definition}

For the next few lemmas fix an $A$ as above.

\begin{lemma}
For every $t \notin A$, there is a unique element of $\{s \in T \; | \; s \sqsubseteq_T t\}$ in $\partial A$.
\end{lemma}

\begin{proof}
Since $t \notin A$ if all of its predecessors are in $A$ then $t \in \partial A$ so we're done. Otherwise, the set $s \in P_t$ not in $A$ is non-empty. Since this set is well-ordered there is a least such element, which must be in $\partial A$. If $s' \in P_t$ above $s$ then $s' \notin \partial A$ since it has a predecessor not in $A$, namely $s$.
\end{proof}

Given this lemma we may define, for each $t \notin A$ the {\em projection} ${\rm proj}_{\partial A} (t)$ to be this unique predecessor. Note that if $t \in \partial A$ then ${\rm proj}_{\partial A}(t) = t$.

\begin{lemma}
For every $s \in A$ there is a (usually not unique) $t \in \partial A$ so that $s \sqsubseteq_T t$.
\end{lemma}

\begin{proof}
Fix $s \in A$. Since $T$ is normal and $A$ is countable, there is an $\alpha \in \omega_1$ so that $A \cap {\rm Lev}_\alpha(T) = \emptyset$ and $s$ is comparable with some $s' \in {\rm Lev}_\alpha(T)$. In particular the collection of $t$ above $s$ not in $A$ is non-empty. Since $T$ is well-founded this set therefore has minimal elements and any such minimal element is in $\partial A$ by definition.
\end{proof}

All of these simple observations and definitions are in the service of defining the following forcing notion, which is key.

\begin{definition}[The forcing notion $\Q_A$]
For $A$ as above define $\Q_A$ to be the forcing notion consisting of pairs $(s, \pi)$ ($\pi$ for {\em promise}) so that the following hold.
\begin{enumerate}
\item
$s \in \P(T)$ with ${\rm dom}(s) \subseteq A$.
\item
$\pi$ is a finite partial function from $\partial A$ to $[\omega]^{<\omega}$
\item
If $t \in {\rm dom}(p)$, $t' \in {\rm dom}(\pi)$, $t\sqsubseteq_T t'$ and $n \in \pi(t')$ then $s(t) \neq n$.
\end{enumerate}
The extension relation is defined by $(s, \pi) \leq (r, \tau)$ if and only if $s \supseteq r$, ${\rm dom}(\pi) \supseteq {\rm dom}(\tau)$ and for every $t \in {\rm dom}(\tau)$ we have $\pi(t) \supseteq \tau(t)$.
\end{definition}

The intuition for this definition is that $\Q_A$ is a subforcing of $\P(T)$ adding $g_T \hook A$. The point is that the second coordinate consists of the information relevant to $g_T \hook A$ given by conditions $p \in G_T$ with ${\rm dom}(p) \nsubseteq A$. To see this, observe that if $p$ is a condition in $\P(T)$ with domain not contained in from $A$, say $p(t) = n$ for some $t \in {\rm dom}(p)\setminus A$ and $n < \omega$, then the ramification for $g_T\hook A$ is that for every $t' \sqsubseteq_T t$ in $A$, we have $p \forces \dot{g}_T(\check{t}') \neq \check{n}$. This is what the promises give us. All of this is more formally expressed by the next lemma.

\begin{lemma}
Let $A$ be a countable, downward closed subset of $T$.
\begin{enumerate}
\item
$\Q_A$ is $\sigma$-centered.
\item
If $H \subseteq \Q_A$ is $V$-generic then $h:=\bigcup_{(s, \pi) \in H} s$ is a specializing function on $A$ in the sense that $h$ maps $A$ to $\omega$ and is injective on chains. Moreover if $b \subseteq A$ is linearly ordered and there is a $t \in T$ greater than every $t' \in b$ there are infinitely many $n \notin {\rm range}(h \hook b)$.
\item
In $V[h]$ define $\mathbb R(h)$ the be the set of finite partial functions $p:T\setminus A \to \omega$ so that $p \cup h$ is a partial specializing function. Let $\dot{\mathbb R}(\dot{h})$ be the $\Q_A$-name for this partial order. Then $\P(T)$ densely embeds into $\Q_A * \dot{\mathbb R}(\dot{h})$ and $g_T \hook A=h$ is the union of the first coordinates of the generic set added by $\Q_A$.
\end{enumerate}
\end{lemma}

For simplicity in what follows if $(s, \pi) \in \Q_A$ call $s$ the {\em stem} of the condition.

\begin{proof}
1. Fix a stem $s$ and let $\pi$ and $\tau$ be so that $(s,\pi)$ and $(s, \tau)$ are both conditions in $\Q_A$. Let $\pi \cup \tau:\partial A \to [\omega]^{<\omega}$ defined by $(\pi \cup \tau) (t) = \pi(t) \cup \tau (t)$ for $t \in {\rm dom}(\pi) \cup {\rm dom}(\tau)$ with $\pi(t) = \emptyset$ for $t \in {\rm dom}(\tau) \setminus {\rm dom}(\pi)$ and idem for $\tau(t)$. It's easy to check that $(s, \pi \cup \tau)$ is a condition in $\Q_A$ extending both $(s, \pi)$ and $(s, \tau)$ so any two conditions with the same stem are compatible. Since there are only countably many stems this proves this first part.

2. This is a straightforward density argument.

3. First note that every element of $\dot{\mathbb R}(\dot{h})$ evaluates to a ground model element (since it is a finite partial function between two ground model sets) so by strengthening if necessary we can decide a check name for every condition in $\dot{\mathbb R}(\dot{h})$ hence $\Q_A * \dot{\mathbb R}(\dot{h})$ has a dense subset of the form $(s, \pi, \check{p})$ with $p \in \P(T)$ and ${\rm dom}(p) \subseteq T\setminus A$. From now on we drop the check on top of $p$. Now there is a further dense subset so that the following hold:

\begin{enumerate}
\item
For every $t \in {\rm dom}(p)$ we have ${\rm proj}_{\partial A} (t) \in {\rm dom}(\pi)$ and $p(t) \in \pi({\rm proj}_{\partial A}(t))$ 
\item
For every $r \in {\rm dom}(\pi)$ and every $n\in \pi(r)$ then there is a $t_r \in {\rm dom}(p)$ with ${\rm proj}_{\partial A} (t_r) = r$ and $p(t_r) = n$
\end{enumerate}

To see this, begin with any $(s, \pi, p) \in \Q_A * \dot{\mathbb R}(\dot{h})$ and first, for every $t \in {\rm dom}(p)$ add $({\rm proj}_{\partial A}(t), p(t))$ to $\pi$. Then for every $(r, n) \in \pi$ if there is not already some $t_r \in {\rm dom}(p)$ with $p(t_r) = n$ then we are done otherwise no node in the domain of $p$ projects to $r$ and has value $n$ under $p$ so we can add a $t_r \notin {\rm dom}(p)$ whose projection is $r$ and assign this node value $n$.  The fact that everything is finite, along with the fact that $\pi(r) \ni n$ precludes $s(t) = n$ for any $t$ below $r$ allows this to work. 

Let $D$ be this dense set of conditions. To summarize, $(s, \pi, p) \in D$ if and only if ${\rm dom}(\pi) = \{{\rm proj}_{\partial A}(t) \; | \; t \in {\rm dom}(p)\}$ and for each $t \in {\rm dom}(\pi)$ we have that $\pi(t) = \{n \; | \; \exists s \in {\rm dom}(p) (t \sqsubseteq_T s \, {\rm and} \, p(s) = n)\}$. We claim that $\P(T)$ is isomorphic to $D$. Indeed, by construction we can read off $\pi$ from $p$ so we can drop the middle coordinate at which point it becomes clear that $p \mapsto (p \hook A, p \hook T\setminus A)$ is the desired isomorphism.
\end{proof}

The point of all of this is the following. Suppose $\dot{x}$ is a $\P(T)$-name for a countable sequence of ordinals. For each $n < \omega$ let $A_n$ be a maximal antichain deciding $\dot{x}(\check{n})$. Note that $A_n$ is countable for all $n < \omega$ by the ccc. Let $A_{\dot{x}}$ be the downward closure of $\bigcup \{{\rm dom}(p) \; | \; p \in \bigcup_{n < \omega} A_n\}$. Observe that all of $\dot{x}$ is decided by $g_T \hook A_{\dot{x}}$: for each $n < \omega$ in $V[g_T]$ there must be a $p \in A_n$ so that $p \subseteq g_T \hook A_{\dot{x}}$ since each $A_n$ is a maximal antichain and the union of the domains of all its elements are contained in $A$. Thus $\dot{x}$ was added by $\Q_{A_{\dot{x}}}$. In particular we have the following.

\begin{theorem}
Every countable set of ordinals added by $\P(T)$ is in a $\sigma$-centered subextension.
\label{sigmacentered}
\end{theorem}

A remark about this theorem is in order.

\begin{remark}
In general whether or not $\P(T)$ itself is $\sigma$-centered depends on the tree and the ambient set theory. For instance, if $T$ is Souslin or, more generally, Baire, then $\P(T)$ is not $\sigma$-centered since it is not absolutely ccc. Conversely, under $\MA$ $\P(T)$ will be $\sigma$-centered for any tree $T$ of cardinality ${<}2^{\aleph_0}$. This is because $\MA$ implies all ccc forcing notions of size less than continuum are $\sigma$-centered, see \cite[Lemma III.3.46]{KenST}. The following question appears to be open however and would be interesting to explore.  

\end{remark}

\begin{question}
Is there a combinatorial condition on $T$ which is equivalent to $\P(T)$ being $\sigma$-centered? For instance, is it possible that $\P(T)$ is $\sigma$-centered if and only if $T$ is special?
\end{question} 

A corollary of this theorem is an alternative proof of one of the main consequences of \cite[Corollary 3.3]{ChoZa15}. Recall that a forcing notion $\P$ is $\sqsubseteq^{\rm Random}$-{\em good} if for every sufficiently large $\theta$ and every countable $M \prec H_\theta$ containing $\P$ if $x \in 2^\omega$ is random over $M$ then $\forces_\P$ ``$x$ is random over $M[\dot{G}]$", see \cite[Chapter 6]{BarJu95}. Clearly if $\P$ is $\sqsubseteq^{\rm Random}$-good then it adds no random reals and hence the $G_\delta$ null sets coded in the ground model are a null covering family in any generic extension by $\P$. It's well know that $\sigma$-centered forcing notions are $\sqsubseteq^{\rm Random}$-good \cite[Theorem 6.5.30]{BarJu95} and finite support iterations of ccc $\sqsubseteq^{\rm Random}$-good forcing notions do not add random reals, \cite[Theorem 6.5.29]{BarJu95}.

\begin{corollary}[Chodounsk\'{y}-Zapletal \cite{ChoZa15}]
\begin{enumerate}
\item
For any generalized Aronszajn tree $T$ the forcing $\P(T)$ is $\sqsubseteq^{\rm Random}$-good and hence adds no random reals, and the set of ground model Borel codes for $G_\delta$ null sets form a covering family.
\item
Let $\gamma$ be an ordinal and $\langle (\P_\alpha, \dot{\Q}_\alpha) \; | \; \alpha < \gamma\rangle$ be a finite support iteration of forcing notions so that for all $\alpha$ $\forces_\alpha$``$\dot{\Q}_\alpha$ is of the form $\P(\dot{T})$ for some generalized Aronszajn tree $\dot{T}$". If $G\subseteq \P_\gamma$ is generic over $V$ then in $V[G]$ we have ${\rm cov}(\Null) = \aleph_1$ and indeed the ground model $G_\delta$ null sets form a covering family.
\end{enumerate}
\end{corollary}

\begin{proof}
Note that 2. follows from 1. by the iteration theorem \cite[Theorem 6.5.29]{BarJu95}. For 1., fix a generalized Aronszajn tree $T$, let $\theta$ be sufficiently large and let $T, \P(T) \in M \prec H_\theta$ with $M$ countable. Let $x \in 2^{\omega}$ be random over $M$. If there is a condition $p \in \P(T)$ forcing that $x$ is not random over $M[\dot{G}]$ then this is because there is a $\P(T)$-name for a Borel code of a null set $\dot{A} \in M$ and $p \forces x \in \dot{A}$. But $\dot{A}$ is added by a $\sigma$-centered forcing notion $\Q_{\dot{A}} \in M$ so $x$ cannot be forced to be in this set.
\end{proof}

\begin{remark}
In our opinion the above proof and the aforementioned proof of Chodounsk\'{y} and Zapletal in \cite{ChoZa15} complement one another as each gives something the other does not. In \cite{ChoZa15} they prove that $\P(T)$ satisfies a condition they term Y-.c.c. and show that Y-.c.c. forcing notions do not add random reals, even when iterated. Every $\sigma$-centered forcing is Y-.c.c. though not every Y-.c.c. forcing is necessarily $\sigma$-centered. Our proof complements this since it shows that nonetheless every real added by $\P(T)$ is already added by a $\sigma$-centered forcing. It bears asking whether this property of $\P(T)$ is actually true more generally of Y-c.c. forcing notions.
\end{remark}

\begin{question}
If $\P$ is Y-c.c. and $\dot{x}$ is a $\P$-name for a countable sequence of ordinals is there a $\sigma$-centered subforcing $\Q_{\dot{x}}$ of $\P$ which adds $\dot{x}$?
\end{question}

Another interesting corollary of Theorem \ref{sigmacentered} is that if $T$ has countable levels then $\partial A_{\dot{x}}$ is countable and hence $\Q_{A_{\dot{x}}}$ is countable. As a result we get the following.
\begin{corollary}
If $T$ has countable levels then every countable set of ordinals added by $\P(T)$ is in a Cohen subextension.
\end{corollary}

I do not know what happens when the tree has uncountable levels.

\begin{question}
If $T$ has uncountable levels is every new countable set of ordinals in a Cohen subextension?
\end{question}

Based on the results of the next section I conjecture that in fact this is the case.

\section{The Special Tree Model}
Now we look at iterating forcing notions of the form $\P(T)$ with an eye towards proving Main Theorem \ref{mainthm1}. First we consider the simpler case where we aim to force $\st = \cc = \kappa$ for a regular cardinal $\kappa$. The techniques used here encompass the main technical contributions of this paper. The more general case described in Main Theorem \ref{mainthm1} where $\st$ and $\cc$ are allowed to be different and singular requires simply adapting the arguments from the regular case to the models constructed in \cite[Theorems 46 \& 47]{Morass}.

\subsection{The Regular Case}
Fix a regular cardinal $\kappa$. We describe what we mean by the {\em special tree model} (of length $\kappa$). Assume $\GCH$ in $V$. Let $T_0$ be a generalized Aronszajn tree, let $\P_0$ be the trivial forcing and let $\P_1$ be $\P(T_0)$ which, formally, we can treat as $\P_0 * \check{\P}(\check{T}_0)$ (so $\dot{\Q}_0$ is this check name). Inductively let $\dot{T}_\alpha$ be a $\P_\alpha$-name for a generalized Aronszajn tree chosen by some suitable bookkeeping and let $\dot{\Q}_\alpha$ be a $\P_\alpha$-name forced to be $\P(\dot{T}_\alpha)$. Let $\dot{g}_\alpha$ be the $\P_{\alpha+1}$-name for the generic specializing function added by $\dot{\Q}_\alpha$. For any $\alpha \leq \beta < \kappa$ let $\P_{\alpha, \beta}$ denote the quotient forcing $\P_\beta / \P_\alpha$. Without loss of generality we assume that each tree has universe a set of ordinals below $\kappa$. The special tree model is now simply $V[G_\kappa] = V[\langle \dot{g}_\alpha^{G_\kappa} \; | \; \alpha < \kappa\rangle]$ for $G_\kappa \subseteq\P_\kappa$ generic over $V$. For the rest of this subsection we fix all of these objects.

\begin{lemma}
In $V[G_\kappa]$ we have $\kappa^{<\kappa} = \kappa = \st = {\rm cov}(\Me) = 2^{\aleph_0}$.
\label{easy}
\end{lemma}

This lemma is essentially standard so we merely sketch the main points.

\begin{proof}
That $\cov(\Me)= 2^{\aleph_0} = \kappa=\kappa^{<\kappa}$ follows from well known facts concerning finite support iterations of ccc forcing notions, see e.g. \cite[p. 81]{BlassHB}. Moreover, since $\kappa$ is regular, every bounded subset appears at some initial stage and, in particular every generalized Aronszajn tree $T \in V[G_\kappa]$ of size ${<}\kappa$ is already in $V[G_\alpha]$ for some $\alpha < \kappa$ where $G_\alpha \subseteq \P_\alpha$ is $V$-generic. It follows that any such tree was specialized (assuming our book keeping device kept the books well enough) so $\st = \kappa$.
\end{proof}

We now turn to the meat of this section: showing that $\non(\Me)$ and several other cardinal characteristics are $\aleph_1$ in the special tree model. For the convenience of the reader we remind them what we aim to prove for Main Theorem \ref{mainthm1} in the case of $\kappa = \st = \cc$ regular.
\begin{theorem}
In the special tree model of length $\kappa$ we have $\st = 2^{\aleph_0} = \cov(\Me)=\kappa$ and $\non(\Me) = \mfa = \mathfrak{s} =\mathfrak{g}= \aleph_1$.
\end{theorem}

In light of Lemma \ref{easy}, in order to prove this theorem it remains to show that in the special tree model $\non(\Me) = \mfa = \mathfrak{s} =\mathfrak{g}= \aleph_1$. This follows collectively from Theorems \ref{nonmeagerset}, \ref{thmas} and \ref{thmg} below. To begin, we reformulate a consequence of the ccc that we have essentially seen for use later.

\begin{lemma}
Let $\dot{x}$ be a $\P(T_0)$ name for an $\omega$ sequence of ordinals and let $p \in \P(T_0)$ be a condition. There is a $\gamma < \omega_1$ so that for all $n < \omega$ there is an $r \leq p$ which decides $\dot{x} \hook n$ and $p \hook {\rm lev}_\beta(T_0) = r \hook {\rm lev}_\beta(T_0)$ for all $\beta \in [\gamma, \omega_1)$.
\label{stage0}
\end{lemma}


\begin{proof}
Fix a $\P(T_0)$-name $\dot{x}$ as in the statement of the lemma. Since $A_{\dot{x}}$ is countable, there is a $\gamma$ greater than the supremum of the levels with nonempty intersection with $A_{\dot{x}}$. Any such $\gamma$ clearly suffices to witness the lemma.
\end{proof}

The above property is essentially preserved iteratively. In full generality this is as follows.

\begin{lemma}
For all $\alpha \leq \beta < \kappa$, all $p \in \P_\alpha$, all $\dot{q} \in \P(\dot{T}_\alpha)$ and every $\P_{\alpha, \beta}$-name for a countable set of ordinals $\dot{x}$ there is a $\gamma < \omega_1$ so that $p$ forces that for all $n < \omega$ there is an $r \in \P_{\alpha, \beta}$ with the following properties: 
\begin{enumerate}
\item
$r$ decides $\dot{x} \hook n$, 
\item
$r(\alpha + 1) \leq \dot{q}$ and 
\item
for all $\xi \in [\gamma, \omega_1)$ we have $r(\alpha + 1) \hook {\rm lev}_\xi(\dot{T}_\alpha) = \dot{q} \hook {\rm lev}_\xi(\dot{T}_\alpha)$.
\end{enumerate}
\label{iteration}
\end{lemma}

\begin{proof}
By induction on $\beta$. The case where $\alpha = \beta$ is essentially the same as Lemma \ref{stage0} so we assume $\alpha < \beta$. Fix $\dot{x}$ as in the hypothesis of the lemma. Let $p \in \P_\alpha$ and let $G_\alpha \subseteq \P_\alpha$ be $V$-generic with $p \in G_\alpha$. Work in $V[G_\alpha]$ and let $T_\alpha = \dot{T}_\alpha^{G_\alpha}$ and $q = \dot{q}^{G_\alpha}$. 

\noindent \underline{Case 1: $\beta = \beta_0 + 1$}. By induction, for every $\P_{\alpha, \beta_0}$-name $\dot{y}$ for a countable set of ordinals there is a $\gamma < \omega_1$ so that for every $n < \omega$ there is an $r \in \P_{\alpha, \beta_0}$ so that $r$ decides $\dot{y} \hook n$, $r(\alpha + 1) \leq q$ and for all $\xi \in [\gamma, \omega_1)$ we have $r (\alpha + 1) \hook {\rm lev}_\xi(T_\alpha) = q \hook {\rm lev}_\xi(T_\alpha)$. 

Consider $q$ as a $\P_{\alpha, \beta_0}$ condition (with support $\{\alpha + 1\}$) and temporarily work in $V[G_{\beta_0}]$ where $G_{\beta_0} = G_\alpha * H$ with $q \in H \subseteq \P_{\alpha, \beta_0}$-generic over $V[G_\alpha]$. Since $\P_\beta = \P_{\beta_0} * \P(\dot{T}_{\beta_0})$, in $V[G_{\beta_0}]$ there is a $\P(\dot{T}_{\beta_0}^{G_{\beta_0}})$-name $\dot{x}'$ forced to be equal to $\dot{x}$. For each $n < \omega$ let $B_n$ be a maximal antichain in $\P( \dot{T}^{G_{\beta_0}}_{\beta_0})$ deciding $\dot{x}'(\check{n})$ and let $B_{\dot{x}'} \subseteq \dot{T}^{G_{\beta_0}}_{\beta_0}$ be the downward closure of the union of the domains of $\bigcup_{n < \omega} B_n$. Note that $B_{\dot{x}'}$ is a countable set of ordinals and each $B_n$ is a countable set of ordinals. Finally let $z$ be a countable sequence of ordinals coding $B_{\dot{x}'}$, each $B_n$ as well as how each $r \in B_n$ decides $\dot{x}'(\check{n})$. Finally, back in $V[G_\alpha]$ let $\dot{z}$ be a $\P_{\alpha, \beta_0}$-name for $z$.

Applying our inductive hypothesis to $\dot{z}$, we can find a $\gamma < \omega_1$ so that for every $n < \omega$ there is an $r \in \P_{\alpha, \beta_0}$ so that $r$ decides $\dot{z} \hook n$, $r(\alpha + 1) \leq q$ and for all $\xi \in [\gamma, \omega_1)$ we have $r (\alpha + 1) \hook {\rm lev}_\xi(T_\alpha) = q \hook {\rm lev}_\xi(T_\alpha)$. But now, given any $n < \omega$ we can decide $\dot{x} \hook n$ simply by first deciding enough of $\dot{z}$ to find compatible conditions in $\P(\dot{T}_{\beta_0})$ in $\{B_i \; | \; i < n\}$ and then finding a strengthening of such conditions with domain included in $B$. In particular for any $n < \omega$ if $r \in \P_{\alpha, \beta_0}$ decides $\dot{z} \hook m$ for some sufficiently large $m < \omega$ then we have enough information to find an $r' \in \P_{\alpha, \beta}$ so that $r' \hook \beta_0 = r$ and $r'$ decides $\dot{x} \hook \check{n}$. Consequently this $\gamma$ is as needed for $\beta$ as well.

\noindent \underline{Case 2: $\beta$ is a limit ordinal}. Since countable sets of ordinals are only added at stages of countable cofinality we may assume that ${\rm cf}(\beta) = \omega$. Let $\{\beta_n\}_{n< \omega}$ be a strictly increasing sequence of ordinals with supremum $\beta$ and let $\beta_0 = \alpha$. By induction, for every $n, k <\omega$ and every $\P_{\alpha, \beta_n}$-name $\dot{y_n}$ for a countable set of ordinals there is a $\gamma_n < \omega_1$ and there is an $r \in \P_{\alpha, \beta_n}$ so that $r$ decides $\dot{y}_n \hook k$, $r( \alpha + 1 )\leq q$ and for all $\xi \in [\gamma_n, \omega_1)$ we have $r (\alpha + 1) \hook {\rm lev}_\xi(T_\alpha) = q \hook {\rm lev}_\xi(T_\alpha)$. Note that the above works uniformly for any $\gamma$ above $\gamma_n$ for all $n < \omega$ simultaneously. 

For each $n < \omega$ let $A_n \subseteq \P_{\alpha, \beta}$ be a maximal antichain deciding $\dot{x} \hook n$. Let $A_{n, k} = \{r \hook \beta_k\; | \; r \in A_n\}$. Note that by finite support for each $n < \omega$ we have $A_n = \bigcup_{k < \omega} A_{n, k}$. For each $n < \omega$ let $\dot{z}_n$ be the $\P_{\alpha, \beta_n}$-name for the countable set of ordinals coding $\bigcup_{l < \omega} \bigcup_{j < n +1} A_{l, j}$ alongside the relevant countable subtrees and the decisions made by the elements of the maximal antichain. Our inductive hypothesis gives us countable ordinals $\{\gamma_n\}_{n < \omega}$ so that for all $n, k < \omega$ we can find an $r \in \P_{\alpha, \beta_n}$ so that $r$ decides $\dot{z}_n \hook k$, $r( \alpha + 1 )\leq q$ and for all $\xi \in [\gamma_n, \omega_1)$ we have $r (\alpha + 1) \hook {\rm lev}_\xi(T_\alpha) = q \hook {\rm lev}_\xi(T_\alpha)$.

Fix an ordinal $\gamma > {\rm sup}_{n < \omega} \, \gamma_n$ and note, as described above $\gamma$ works concurrently for all $\dot{z}_n$. It follows though that for this $\gamma$ for any $k < \omega$ we can always find an $r$ so that $r$ decides $\dot{z} \hook k$, $r( \alpha + 1 )\leq q$ and for all $\xi \in [\gamma, \omega_1)$ we have $r (\alpha + 1) \hook {\rm lev}_\xi(T_\alpha) = q \hook {\rm lev}_\xi(T_\alpha)$, which completes the proof.
\end{proof}

The main theorem is a corollary of this lemma and the following one.

\begin{lemma}
Let $T$ be an Aronszajn tree, $\{\alpha_\xi \; | \; \xi < \omega_1\}$ enumerate (in order) the infinite levels of $T$ (if all but countably many levels are finite then of course $T$ has an uncountable branch), and let $A_\xi \subseteq {\rm lev}_{\alpha_\xi}(T)$ be countably infinite, say $A_\xi = \{t^\xi_n \; | \;n < \omega\}$. If $g_T$ is $\P(T)$-generic then the function $c_\xi \in 2^\omega$ defined by $n \mapsto g_T(t^\xi_n) \, {\rm mod} \, 2$ is Cohen generic over $V$.
\label{cohenreals}
\end{lemma}

Slightly more succinctly, and less formally, this lemma states that the parity of the restriction of $g_T$ to any countably infinite subset of $T$ is a Cohen real.

\begin{proof}
Let $p \in \P(T)$ be a condition, $\xi < \omega_1$ and let $c^p_\xi$ be the finite function determined by $n \mapsto p(t^\xi_n) \, {\rm mod} \, 2$. Let $D \subseteq 2^{<\omega}$ be dense open and let $c' \supseteq c^p_\xi$ be an extension into this set. For every $k \in {\rm dom}(c') \setminus {\rm dom}(c^p_\xi)$ there are only finitely many $s \leq t_k^\xi$ in the domain of $p$ hence we can extend $p$ to a condition with $t_k^\xi$ in the domain and map it to either an even or odd number as we choose. But then we can extend $p$ to a $q$ so that $c_\xi^q = c'$ as needed.
\end{proof}

Note the same proof essentially shows that the $c_\xi$'s are mutually generic though we will not need this.

\begin{theorem}
Let $\{\alpha_\xi \; | \; \xi < \omega_1\}$ be as above for $T = T_0$. Let $G_\kappa \subseteq \P_\kappa$ be generic over $V$. Borrowing the notation from Lemma \ref{cohenreals}, in $V[G_\kappa]$ the set $\{c_\xi \; | \; \xi < \omega_1\}$ forms a non meager set and hence $\non(\Me) = \aleph_1$ in $V[G_\kappa]$.
\label{nonmeagerset}
\end{theorem}

To ease the notation, below we assume $\xi = \alpha_\xi$ for all $\xi < \omega_1$. This causes no loss of generality and prevents having to add subscripts unnecessarily. Also, given a condition $s \in \P(T_0)$ and a countable ordinal $\xi < \omega_1$ let $c_\xi^s$ be the corresponding Cohen condition determined by $s$ for $c_\xi$ as used in the proof of Lemma \ref{cohenreals}.

\begin{proof}
We need the following fact.

\begin{fact}[Theorem 2.2.4 of \cite{BarJu95}]
If $A \subseteq 2^\omega$ then $A$ is meager if and only if there is a strictly increasing $f \in \baire$ and an $x_f \in \cantor$ so that $$A \subseteq \{y \in 2^\omega \; | \; \exists j \in \omega \, \forall l > j \, y\hook[f(l), f(l+1)) \neq x_f\hook [f(l), f(l+1))\}$$
\label{meagerfact}
\end{fact}

Thus we need to show that there is no $f$ and $x_f$ as described above for $A = C:=\{c_\xi\; | \; \xi < \omega_1\}$. Let $\dot{x}_f, \dot{f}$ be names for, respectively, an element of Cantor space and a strictly increasing function from $\omega$ to $\omega$. Fix $p \in \P_\kappa$. By Lemma \ref{iteration}, there is a $\gamma < \omega_1$ so that for all $\xi \in [\gamma, \omega_1)$ and all $k < \omega$ we can find an $r \leq p$ with $r$ deciding $\dot{f} \hook k+2$ and $\dot{x}_f \hook f(k+2)$ and $r(1) \hook {\rm Lev}_\xi(T_0) = p(1) \hook {\rm Lev}_\xi(T_0)$. Now for every $l < \omega$ let $k > {\rm max} \{l, {\rm sup} \{n < \omega\; | \exists \xi < \omega_1 \,  t^\xi_n \in {\rm dom}(p(1))\}\}$. Note that such a $k \in \omega$ exists since ${\rm dom}(p(1))$ is finite\footnote{Note that $p(1) \in \P(T_0)$ (as opposed to $p(0)$) since $\P_0$ is defined to be the trivial forcing and $\P_1$ is $\P(T_0)$.}. Now find an $r_k$ as described above deciding $\dot{f} \hook k+2$ and $\dot{x}_f \hook \dot{f}(k+2)$. Note that for all $\xi \in [\gamma, \omega_1)$ we have ${\rm dom}(c_\xi^{r(1)}) \subseteq k$ since $r(1) \hook {\rm Lev}_\xi(T_0) = p(1) \hook {\rm Lev}_{\xi}(T_0)$ and by assumption the latter is contained in $k$. In particular ${\rm dom}(c_\xi^{r(1)})$ does not contain anything in the interval $[\dot{f}(k), \dot{f}(k+1))$. Therefore, we can extend $r(1) \hook {\rm Lev}_\xi(T_0)$, in the same way described in Lemma \ref{cohenreals} so that the part of $c_\xi$ decided agrees with $\dot{x}_f$ on this interval. This means in particular that for every $l < \omega$ it is dense to force that there is a $k > l$ so that $c_\xi\hook [\dot{f}(k), \dot{f}(k+1)) = \dot{x}_f \hook [\dot{f}(k), \dot{f}(k+1))$ for a tail of $\xi$ and therefore no meager set of the form described in Fact \ref{meagerfact} can capture all the Cohen reals in $V[G_\kappa]$ which implies that the set is non meager as needed.
\end{proof}

\begin{remark}
A well known result of Miller \cite[Theorem 2.4.7]{BarJu95} states that $\non(\Me)$ is the least size of a set $A \subseteq \baire$ so that no $g \in \baire$ is eventually different from every $f \in A$. A simple tweaking of the argument above shows that the Cohen generics in Baire space coded by the $c_\xi$'s form such a family and hence give an alternative proof of Theorem \ref{nonmeagerset}.
\end{remark}

A very similar proof shows that $\mfa = \mfs =  \aleph_1$ in $V[G_\kappa]$.

\begin{theorem}
\begin{enumerate}
\item
There is a tight MAD family of size $\aleph_1$ in $V[G_\kappa]$ and, in particular $\mfa=\aleph_1$ in the special tree model. 

\item
$\mfs = \aleph_1$ in the special tree model.

\end{enumerate}
\label{thmas}
\end{theorem}

Recall here that an infinite MAD family $\mathcal A$ is {\em tight} if for every countable set $\{A_n \; | \; n < \omega\} \subseteq \mathcal I^+(\mathcal A)$ there is a $B \in \mathcal I(\mathcal A)$ which intersects each $A_n$ infinitely often\footnote{If $\mathcal A$ is an almost disjoint family, $\mathcal I(\mathcal A)$ is the ideal generated by $\mathcal A$ i.e. the set $\{B \subseteq \omega \; | \; \exists B_0, ..., B_{n-1} \in \mathcal A \, B \subseteq^* \bigcup_{i < n} B_i\}$.}. Clearly tightness implies maximality. See \cite{orderingMAD, restrictedMADfamilies} for more on this notion.

\begin{proof}
The proof of parts one and two are very similar to each other and to their corresponding proofs in the case of the Cohen model. They also all apply the same application of Lemma \ref{iteration} as used in the proof of Theorem \ref{nonmeagerset} so we merely sketch them and leave the details to the interested reader. In essence in each case it is well known that in the Cohen model the first $\aleph_1$-many Cohen reals code a tight MAD family (respectively a splitting family) of size $\aleph_1$ and we show that the $\aleph_1$ Cohen reals added by specializing the first tree play an almost identical role in the special tree model.

1. We continue with the same notation as in Theorem \ref{nonmeagerset}. Work in $V[G_\kappa]$. For each Cohen real $c_\xi$ let $d_\xi \in \baire$ be defined by $d_\xi (l) = k$ if and only if the $l^{\rm th}$ block of $1$'s in $c_\xi$ has length $k$. In other words, if $c_\xi$ starts out as $11101100110$ then $d_\xi(0) = 3$, $d_\xi(1) = 2$, $d_\xi(2) = 0$ and $d_\xi(3) = 2$. It's well known that $d_\xi$ described this way is also Cohen over $V$. This exact coding method is not so important, we just need for each $\xi< \omega_1$ a Cohen generic in $\baire$ coded by $c_\xi$.

Fix an infinite partition of $\omega$ into infinite pieces, $\{A_n \; | \; n < \omega\} \in V$. For each $\xi \in [\omega, \omega_1)$ inductively define $A_\xi$ as follows. First, rearrange $\{A_\zeta \, | \, \zeta < \xi\}$ into an $\omega$ sequence so that each element appears infinitely often, say $\{B_n \; | \; n < \omega\}$. Now, define an infinite $A_\xi = \{a_n \; | \; n <\omega\} \subseteq \omega$ inductively by letting $a_0$ be the $d_\xi(0)^{\rm th}$ element of $\omega\setminus B_0$, $a_{n+1}$ be the $d_\xi(n+1)^{\rm th}$ element of $\omega \setminus (\bigcup_{l < n+1} B_l \cup \{a_0, ..., a_n\})$. Obviously the family $\mathcal A = \{A_\xi \; | \; \xi < \omega_1\}$ forms an almost disjoint family. We have to see that it is tight and hence maximal. Let $\{D_n \; | \; n < \omega\}$ be an infinite sequence of elements of $[\omega]^\omega$ and let $\dot{D}$ be a $\P_\kappa$ name for the subset of $\omega^2$ coding them all. Fix $p \in \P_\kappa$. By Lemma \ref{iteration} we can decide any finite part of $\dot{D}$ by strengthening $p$ in a way that leaves a tail of levels of $p(0)$ unperturbed. In particular, if every $D_n$ is is forced to not be almost covered by some finite set of $A_\xi$'s then we can make a tail of the $A_\xi$'s infinitely often equal to all of them using the same argument more or less as in Theorem \ref{nonmeagerset}.

2. For the case of $\mfs$, the argument is almost the same. For each $\xi < \omega_1$ let $C_\xi \in [\omega]^\omega$ be the infinite set whose characteristic function is $c_\xi$ i.e. $n \in C_\xi$ if and only if $c_\xi(n) = 1$. Applying essentially the same argument as in Theorem \ref{nonmeagerset} to the set $\{C_\xi | \; \xi < \omega_1\}$ shows that this family is in fact a splitting family in the special tree model.
\end{proof}

Finally we show that an argument similar to the corresponding one for the Cohen model \cite[pp.~25-26]{Blass87} gives that $\mathfrak{g} = \aleph_1$ in the special tree model.

\begin{theorem}
$\mathfrak{g} = \aleph_1$ in the special tree model.
\label{thmg}
\end{theorem}

\begin{proof}
Recall that a set $A \subseteq 2^\omega$ is {\em almost Turing Cofinal} if there is an $x \in 2^\omega$ so that every $y \in 2^\omega$ can be computed by some $a \oplus x$ for $a \in A$ where $\oplus$ denotes the Turing join. Blass showed in \cite[Theorem 2]{Blass87} that if $\kappa < \mathfrak{g}$ then if $\bigcup_{\alpha < \kappa} X_\alpha$ is almost Turing cofinal then some $X_\alpha$ is almost Turing cofinal. Therefore it suffices to show that in the special tree model there is a cover of $2^\omega$ in $\aleph_1$ many pieces so that none of them are almost Turing cofinal. Towards this, for any $x \in 2^\omega$ let $\Phi^x_e$ denote the $e^{\rm th}$-Turing program with oracle $x$ relative to some fixed coding. Note that the program instructions are not dependent on $x$ (though of course the outcome may be). 

For each $x \in V[G_\kappa] \cap 2^\omega$ let $\dot{x}$ be a nice $\P_\kappa$-name for it and let for each $n < \omega$ $A_n(\dot{x})$ be a maximal antichain deciding $\dot{x}(\check{n})$. Let $\gamma_x$ be the least countable ordinal $\gamma$ so that $\bigcup_{n < \omega} \bigcup \{p(1) \; | \; p \in A_n(\dot{x})\} \cap {\rm Lev}_{\gamma}(T_0)  = \emptyset$. In words, $\gamma_x$ is the least countable ordinal $\gamma$ so that we can decide any finite amount of $\dot{x}$ without appealing to conditions whose first nontrivial coordinate has a domain intersecting the $\gamma^{\rm th}$-level of $T_0$. Such a $\gamma$ exists for each $x$ by the ccc (see also Lemma \ref{stage0}). For each $\gamma < \omega_1$ let $X_\gamma = \{x \; | \; \gamma_x < \gamma\}$. Observe that if $\gamma_0 < \gamma_1$ then $X_{\gamma_0} \subseteq X_{\gamma_1}$ essentially by definition. Moreover $\bigcup_{\gamma < \omega_1} X_\gamma = 2^\omega$ and therefore the union of the $X_\gamma$'s is, in particular, almost Turing cofinal. If $\mathfrak{g} > \aleph_1$ then there is a $\gamma < \omega_1$ so that $X_\gamma$ is almost Turing cofinal. Note that this implies that actually a tail of $X_\gamma$'s are almost Turing cofinal since the sets are increasing and obviously any superset of an almost Turing cofinal set is almost Turing cofinal. In other words, for a tail of $\gamma$ there is a $y_\gamma \in 2^\omega$ so that $\{x \oplus y_\gamma \; | \; x \in X_\gamma\}$ is cofinal in the Turing degrees. Since $y_\gamma$ itself is in some $X_\alpha \supseteq X_\gamma$ we can conclude that there is an $\alpha < \omega_1$ so that $\{x \oplus y \; | \; x, y \in X_\alpha\}$ is actually Turing cofinal (no almost). But this is absurd since for any $\xi > \alpha$, any $x, y \in X_\alpha$ and any $e, k < \omega$ we can decide $\Phi_e^{x \oplus y}(k)$ without specializing any piece of $T_0$ at level $\xi$ and therefore $c_\xi$ cannot be the output of $\Phi_e^{x \oplus y}$. With this contradiction the proof is complete.
\end{proof}

Before moving to the next section let us note that the proofs of Theorems \ref{nonmeagerset}, \ref{thmas} and \ref{thmg} did not use having a particularly adept book keeping device, nor the regularity of $\kappa$ or even $\GCH$ (these were used to show $\st = \kappa$). Indeed what the proofs of Theorem \ref{nonmeagerset}, \ref{thmas} and \ref{thmg} show is simply that iterating forcing notions of the form $\P(T)$ with finite support (even if ``iteration" just means once) will force the cardinals $\non(\Me)$, $\mfa$, $\mfs$ and $\mathfrak{g}$ to be $\aleph_1$. As such we have actually shown the following.
\begin{theorem}
Let $0 < \alpha$ be an ordinal and let $\langle \P_\gamma, \dot{\Q}_\gamma \; | \; \gamma < \alpha\rangle$ be a finite support iteration so that for all $\gamma < \alpha$ we have $\forces_\gamma$``$\dot{\Q}_\gamma$ is of the form $\P(\dot{T})$ for some tree $\dot{T}$ of height $\omega_1$ with no cofinal branch". If $G\subseteq \P_\alpha$ is $V$-generic then $V[G] \models \non(\Me) = \mfa = \mfs = \mathfrak{g} = \aleph_1$.
\label{aleph1cards}
\end{theorem}

Note that same conclusion holds if we simply add uncountably many mutually generic Cohen reals, again suggesting that the Question 3 may have a positive answer.

\subsection{The Singular Case}
Now we aim to prove the more general version of Main Theorem \ref{mainthm1}. For the rest of this subsection fix cardinals $\aleph_1 \leq \lambda \leq {\rm cf}(\mu) \leq \mu$ with $\lambda$ (and of course $\mu$) of uncountable cofinality. We will show the following, which, when coupled with Fact \ref{kfact1} below, is simply a sharpening of Main Theorem \ref{mainthm1}.

\begin{theorem}
Assume $\lambda^\omega = 2^\omega = \lambda$ and there is a neat, stationary $(\omega_1, \lambda)$-semimorass (see below for a definition). There is a finite support iteration of the form $\langle \P_\gamma, \dot{\Q}_\gamma \; | \; \gamma < \mu\rangle$ so that for all $\gamma < \mu$ we have $\forces_\gamma$``$\dot{\Q}_\gamma$ is of the form $\P(\dot{T})$ for some tree $\dot{T}$ of height $\omega_1$ with no cofinal branch" and $\P_{\mu}$ forces that $\non(\Me) = \mfa = \mfs = \mathfrak{g} = \aleph_1 \leq \st = \lambda \leq \cc = \mu$.
\label{singular}
\end{theorem}

For the needlessly curious reader we give the definition of a neat, stationary $(\omega_1, \lambda)$-semimorass below, however we will not need it and rather, the proof of Theorem \ref{singular} merely blackboxes some relevant facts from \cite{Morass} due to Koszmider.

\begin{definition}[Neat, Stationary $(\omega_1, \lambda)$-Semimorass, See Definition 1 of \cite{Morass}]
An $(\omega_1, \lambda)$-{\em semimorass} is a family $\mathscr F \subseteq [\lambda]^{\leq\omega}$ satisfying the following conditions.
\begin{enumerate}
\item
$\mathscr F$ is well-founded with respect to inclusion. Let $rank:\mathscr F \to \mathsf{ORD}$ be the corresponding rank function.
\item
For all $X \in \mathscr F$ the set $\mathscr F \hook X : = \{Y \in \mathscr F \; | \; y \subseteq X\}$ is countable.
\item
If $X, Y \in \mathscr F$ have the same rank then they have the same order type (as subsets of $\lambda$) and $\mathscr F \hook Y = \{f_{X, Y}''Z \; | \; Z \in \mathscr F \hook \}$ where $f_{X, Y}:X \to Y$ denotes the unique order isomorphism between $X$ and $Y$.
\item
$\mathscr F$ is directed, i.e. for all $X, Y \in\mathscr F$ there is a $X \in \mathscr F$ so that $X, Y \subseteq Z$.
\item
For each $X$ either $\mathscr F \hook X$ is directed or there are $X_1, X_2 \in \mathscr F$ of the same rank so that $X = X_1 \cup X_2$, $f_{X_1, X_2} \hook X_1 \cap X_2$ is the identity on $X_1 \cap X_2$ and $\mathscr F \hook X = \mathscr F\hook X_1 \cup \mathscr F \hook X_2 \cup \{X_1, X_2\}$.
\item
$\bigcup \mathscr F = \lambda$.
\end{enumerate}
If $\mathscr F$ is an $(\omega_1, \lambda)$-semimorass then we say that $\mathscr F$ is {\em neat} if for every $X \in \mathscr F$ either $rank(X) = 0$ or $X =\bigcup(\mathscr F \hook X)$. Finally we say that such a $\mathscr F$ is {\em stationary} if it is stationary as a subset of $[\lambda]^{\leq \omega}$ i.e. it intersects every $\subseteq$-closed $\subseteq$-unbounded family of countable subsets of $\lambda$.
\end{definition}

We need the following two facts.
\begin{fact}[See Theorem 3 of \cite{Morass}]
If $\GCH$ holds there is a ccc forcing notion $\P$ forcing $\lambda^\omega = 2^\omega = \lambda$ and ``there is a neat, stationary $(\omega_1, \lambda)$-semimorass".
\label{kfact1}
\end{fact}

Given a stationary $(\omega_1, \lambda)$-semimorass $\mathscr F$ denote by $T(\mathscr F)$ the set 
\begin{center}
$\{a \subseteq\mathscr F \; |$ $a$ is a continuous, well-ordered by inclusion chain and $\bigcup a = a\}$ 
\end{center}
(see \cite[Definition 36]{Morass}). Ordered by end extension $T(\mathscr F)$ is a non-special generalized Aronszajn tree by \cite[Fact 39]{Morass}. Note that if $\lambda^\omega = 2^\omega = \lambda$ then $|T(\mathscr F)| = \lambda$. 

\begin{fact}[See Theorem 46 of \cite{Morass}]
Suppose $\lambda^\omega = \cc = \lambda$ and $\mathscr F$ is a stationary $(\omega_1, \lambda)$-semimorass. Let $\langle \P_\gamma, \dot{\Q}_\gamma \; | \; \gamma < \alpha\rangle$ be a finite support iteration so that for all $\gamma < \alpha$ we have $\forces_\gamma$``$\dot{\Q}_\gamma$ is of the form $\P(\dot{T})$ for some tree $\dot{T}$ of height $\omega_1$ and cardinality ${<}\lambda$ with no cofinal branch". Then $\P_\alpha$ forces that $\check{T}(\check{\mathscr F})$ is not special.
\label{kfact2}
\end{fact}

The proof of Theorem \ref{singular}, and hence Main Theorem \ref{mainthm1} is now more or less immediate.
\begin{proof}[Proof of Theorem \ref{singular}]
Assume $\lambda^\omega = 2^\omega = \lambda$ and there is a neat, stationary $(\omega_1, \lambda)$-semimorass $\mathscr F$. By Fact \ref{kfact1} this is possible and hence the assumption (and the Theorem) are not vapid. Via a book keeping device force with an iteration of the form $\langle \P_\gamma, \dot{\Q}_\gamma \; | \; \gamma < \mu\rangle$ so that for all $\gamma < \mu$ we have $\forces_\gamma$``$\dot{\Q}_\gamma$ is of the form $\P(\dot{T})$ for some tree $\dot{T}$ of height $\omega_1$, cardinality ${<}\lambda$ with no cofinal branch". Fact \ref{kfact2} implies that $|(T(\mathscr F))^V| = \lambda$ is non special so $\st \leq \lambda$ while Theorem \ref{aleph1cards} implies that $\non(\Me) = \mfa = \mfs = \mathfrak{g} = \aleph_1$. Since the iteration is a finite support iteration of ccc forcing notions of length $\mu$ and every iterand is forced to have size ${<}\mu$ (and $\mu^{<\mu} = \mu$) we easily get via a nice-name argument that $2^\omega = \mu$ and, moreover, that $\mu$-many mutually generic Cohen reals were added so $\cov(\Me) = \mu$ as well. The only thing left to check is that $\lambda \leq \st$ i.e. that we can build a book keeping device good enough to catch all possible trees of size ${<} \lambda$. For this it is enough to show that there are only at most $\mu$ many nice names for generalized Aronszajn trees of size ${<}\lambda$. This follows from the fact that $\lambda \leq {\rm cf}(\mu)$. See the proof of \cite[Theorem 47]{Morass} for more details on this part of the argument.
\end{proof}

The one case left open by this theorem is the possibility of ${\rm cf}(\cc) < \st$. In the proof above, this is not possible a priori to arrange since for any $\xi \geq cf(\cc)$ we have $\mu^\xi > \mu$ many nice names for subsets of $\xi$ and hence there may very well be too many trees of size $\xi$ to be able to specialize them all by a ccc forcing of size $\mu$. At the same time, forcing with a larger forcing notion however may add $\mu^+$ many reals. We do not know whether these technical issues can be avoided though it is consistent that ${\rm cf}(\cc) < \st$ as witnessed by Laver's Theorem from \cite{Laver1987}, proved as Theorem \ref{Laver} above. However it is unclear how the inequality ${\rm cf}(\cc) < \st$ affects other cardinal characteristics since it is not clear how to iterate to obtain it. As such the following seems deceptively interesting.

\begin{question}
What consequences does the inequality ${\rm cf}(\cc) < \st$ have on other cardinal characteristics and on trees? For instance does ${\rm cf}(\cc) < \st$ imply that $\st$ is regular?
\label{Qcof}
\end{question}


\section{Conclusion and Open Questions}
The (not even so) observant reader will notice that Main Theorem \ref{mainthm1} actually shows that the cardinal characteristics of the Special Tree model are exactly the same as in the Cohen model (other than $\st$ of course, which is $\aleph_1$ since there is always a Souslin tree in the Cohen model). This gives credence to the conjecture in Section 3 that every new real in a $\P(T)$-extension is already in a Cohen subextension. Of course these models are very different combinatorially. Indeed it's well know that in the Cohen Model there are Souslin trees and in particular $\st = \aleph_1$. More interestingly recall that the combinatorial principle $\diamondsuit ({\rm non}(\Me))$ from \cite{DHM04} holds\footnote{The precise definition of $\diamondsuit({\rm non}(\Me))$ is somewhat technical and since we do not need it we omit it here. For the reader who does not know about parametrized diamond principles we simply note that $\diamondsuit({\rm non}(\Me))$ strengthens the statement ${\rm non}(\Me) = \aleph_1$ in a similar way to how $\diamondsuit_{\omega_1}$ strengthens $\CH$.} in the Cohen model and indeed in every iterated model of $\non(\Me) = \aleph_1$ with ``definable" iterands (see \cite[Theorem 6.6]{DHM04} for a more precise statement of this result). However $\diamondsuit({\rm non}(\Me))$ implies there is a Souslin tree (\cite[Theorem 4.7]{DHM04}), and hence must fail in the special tree model. In our view the special tree model is a relatively natural model witnessing ${\rm non}(\Me) = \aleph_1 \land \neg \diamondsuit({\rm non}(\Me))$. We ask how else the Cohen and special tree models can be separated.

\begin{question}
What other well studied principles distinguish the Cohen model and the special tree model? 
\end{question}

On a more basic level, the results above do not address how $\st$ relates to the right hand side of the Cicho\'{n} diagram. However, this is partially addressed by Laver's Theorem \ref{Laver} above (in its slightly modified form). In that model $\mfd = \st$ while in the special tree model (or any length) we have $\st \leq {\rm cov}(\Me)$ and hence $\st \leq \mfd$. The following is therefore open. 

\begin{question}
Does $\ZFC$ prove $\st \leq \mfd$?
\end{question}

We conjecture this is the case. What this would amount to in practice is that there is no way to specialize all generalized Aronszajn trees of size $\aleph_1$ in an iterable way without adding unbounded reals. This question is interesting in its own right.

\begin{question}
Under which conditions on a generalized Aronszajn tree $T$ is there a proper forcing notion $\Q_T$ which specializes $T$ without adding unbounded reals? Cohen reals? What if we drop ``proper"?
\end{question}

If we replace ``not adding unbounded (or Cohen) reals" with ``not adding reals at all" then more is known about this question. Indeed Jensen, under $V=L$ and, Shelah, with no additional hypothesis, have shown that for any (thin) Aronszajn tree $T$ there is a proper forcing notion specializing $T$ without adding reals and this is iterable, see \cite[Chapter V, Theorem 6.1]{PIP}. As a result they obtain models of $\CH \land \neg \diamondsuit$. In \cite{Switz20} the author investigated when ``Aronszajn tree" could be replaced by ``generalized Aronszjan tree of cardinality $\omega_1$" and gave a sufficient combinatorial condition on trees for modifying Shelah's aforementioned forcing to the wide case. However we cannot hope to find for any generalized Aronszajn tree of size $\aleph_1$ a stationary set preserving forcing notion specializing it without adding reals under $\CH$ since the tree $T(S)$ discussed in Proposition \ref{T(S)} cannot be made special without either adding reals or killing the stationarity of $S$.

\begin{question}
Are there (consistently) generalized Aronszajn trees $T$ so that in any forcing extension in which $T$ is special there is a Cohen real over the ground model?
\end{question}

\end{document}